\documentclass[a4paper, 12pt]{article}
\usepackage{fullpage}

\usepackage{mathtext}               
\usepackage[T2A]{fontenc}           
\usepackage[english]{babel}

\usepackage{amsfonts,amssymb,amsthm}
\usepackage{amsmath}
\usepackage{icomma} 

\usepackage{euscript}   
\usepackage{mathrsfs} 

\usepackage{graphicx}  
\graphicspath{{images/}{images2/}}  
\usepackage{wrapfig}

\usepackage{array,tabularx,tabulary,booktabs} 
\usepackage{longtable}  
\usepackage{multirow} 
\usepackage{microtype}

\usepackage{enumitem}
\usepackage{epigraph}

\usepackage{algpseudocode}

\usepackage{setspace} 
\usepackage{hyperref}
\hypersetup{
    colorlinks, urlcolor={black},
    linkcolor={black}, citecolor={black}, filecolor={black},
    pdfauthor={Andrey Sergunin},
    pdftitle={On the number of intersection points of lines and circles in R^3}
}

\newtheorem{theorem}{Theorem}
\newtheorem{lemma}[theorem]{Lemma}	
\newtheorem{corollary}[theorem]{Corollary}

\newtheorem{proposition}[theorem]{Proposition}

\theoremstyle{remark}
\newtheorem*{remark}{Remark}

\DeclareMathOperator{\chr}{char}

\newcommand{\PP}{P}

\title{On the number of intersection points\\ of lines and circles~in~\(\mathbb R^3\)}
\author{Andrey Sergunin\thanks{Laboratory of Combinatorial and Geometric Structures, Moscow Institute of Physics and Technology (National Research University), Russia}}
\date{}

\begin{document}
		
\maketitle

\begin{abstract}
We consider the following question: Given $n$ lines and $n$ circles in $\mathbb{R}^3$, what is the maximum number of intersection points lying on at least one line and on at least one circle of these families. 
We prove that if there are no $n^{1/2}$ curves (lines or circles) lying on an algebraic surface of degree at most two, then the number of these intersection points is $O(n^{3/2})$.
\end{abstract}


\section*{Introduction}

It is easy to see that \(n\) distinct lines in \(\mathbb{R}^3\) can have a quadratic number of intersection points. 
For instance, consider \(n\) lines in general position in the plane. 
Another example is the following. Consider $n$ lines lying on the surface of a hyperboloid with one sheet: $\lfloor n/2 \rfloor$ of these lines belong to one family of generators, $\lceil n/2 \rceil$ of these lines belong to another. 
Note that in this example hyperboloid with one sheet can be replaced by any regulus, that is, the surface spanned by all lines that meet three pairwise skew lines in $\mathbb R^3$. 
A non-trivial upper bound on the number of intersection points of $n$ lines is proven in the following theorem provided that there is no surface of small degree containing many lines. 

\begin{theorem} [Guth--Katz \cite{guth-katz}] \label{Guth-Katz} 
Let $\mathcal{L}$ be a collection of $n$ lines in $\mathbb{R}^3$. 
Let $A \geq 100 n^{1/2}$ and suppose that there are at least $100An$ points incident to at least two lines of $\mathcal{L}$. 
Then there exists a plane or regulus $Z \subset \mathbb{R}^3$ that contains at least $A$ lines from $\mathcal{L}$. 
\end{theorem}

In \cite{guth-katz} Theorem \ref{Guth-Katz} was applied to prove the lower bound in old difficult Erd\H os' problem about the number of distinct distances between points on the plane.

A more general result was proven for curves of arbitrary degree in \cite{guth-zahl}. 

\begin{theorem} [Guth--Zahl \cite{guth-zahl}] \label{Guth-Zahl}
Let $D > 0$. 
Then there are constants $c_1, C_1, C_2>0$ so that the following holds.
Let $k$ be a field and let $\mathcal{L}$ be a collection of $n$ irreducible curves in $k^3$ of degree at most $D$. 
Suppose that $\chr(k) = 0$ or $n \leq c_1 (\chr(k))^2$.
Then for each $A \geq C_1n^{1/2}$, either there are at most $C_2An$ points in $k^3$ incident to two or more curves from $\mathcal{L}$, or there is an irreducible surface $Z$ of degree at most $100D^2$ that contains at least $A$ curves from $\mathcal{L}$.

\end{theorem} 

In the current paper we prove an analogous result about the number of intersection points between lines and circles in $\mathbb{R}^3$ claiming that an irreducible surface $Z$ in this case is either a plane or a hyperboloid with one sheet.

For a collection $\mathcal L$ of lines in $\mathbb R^3$ and a collection $\mathcal C$ of circles in $\mathbb R^3$, denote by $\PP(\mathcal{L}, \mathcal{C})$  the set of points lying on at least one line of $\mathcal{L}$ and at least one circle of $\mathcal{C}$.

\begin{theorem} \label{lines-circles}
Let $\mathcal{L}$ be a collection of $n$ lines and $\mathcal{C}$ be a collection of $m$ circles in $\mathbb{R}^3$.
Then for $A \ge 10^5 \min(n, m)^{1/2}$, either $\PP(\mathcal{L}, \mathcal{P}) \le 1000A (n + m)$ or there is a plane or a hyperboloid with one sheet containing at least $A$ curves of $\mathcal{L} \cup \mathcal{C}$.
\end{theorem}

\begin{remark}
It is known that a hyperboloid with one sheet in $\mathbb{R}^3$ can be defined in some Cartesian coordinate system by an equation \( a x^2 + b y^2 - c z^2 = r \), where $a, b, c, r > 0$.
If $c = 0$ and $a, b, r > 0$, then this equation defines elliptic cylinder, and if $r = 0$ and $a, b, c > 0$, then it defines elliptic cone. 
In the current paper, we consider elliptic cones and cylinders as hyperboloids with one sheet.
Each of these surfaces contains infinite families of circles and lines in $\mathbb{R}^3$.
\end{remark}

In \( \mathbb{C}^3 \) there are surfaces of degree four containing families of generating lines and circles; we refer a reader to the definition of complex circle given after Corollary \ref{bound on number of lines}. 
One of the examples is the surface defined by the equation \((x^2 + y^2 + z^2)^2 + (x + iy)^2 - z^2 = 0\) (see \cite[Example~3.6]{nilov-skopenkov}).
It can be parametrised in \( \mathbb{CP}^3 \) as \( t^2 - 1 : i (t^2 - 1 - 2st) : s(t^2 + 1) : s(t^2 - 1) + 4t \). This surface contains the family of lines \(t = \mbox{const} \) and the family of circles \( s = \mbox{const} \).
Any $n$ lines and $m$ circles on this surface have $nm$ intersections.
Thus Theorem \ref{lines-circles} is incorrect in $\mathbb{C}^3$.  


Our proof is based on two ideas. 
The first one is to find a surface of small degree containing all curves of our collections.
This stage is similar to the idea developed by Guth and Katz in~\cite{guth-katz} to prove Theorem~\ref{Guth-Katz}.
The second one is to show that each irreducible component of the surface constructed in the previous step is ruled and has many circles on it lying in parallel planes. These properties of the irreducible components makes it possible to claim that each of them has degree at most two.
In the similar way this idea is presented in the work~\cite{nilov-skopenkov} of Nilov and Skopenkov, where the authors proved that if through each point of a surface in \(\mathbb{R}^3\) one can draw both a straight line segment and a circular arc, then this surface is a part of algebraic surface of degree at most two. 

The paper is organised as follows. 
In Section \ref{preliminaries} we state auxiliary facts from algebraic geometry. 
Next, in Section \ref{proof of theorem} we show Theorem \ref{lines-circles} using Proposition~\ref{main proposition}, which is proved in Section \ref{proof of propositions}.

\section{Preliminaries} \label{preliminaries}


Let $\mathbb{CP}^3$ be the three-dimensional complex projective space with homogeneous coordinates $x : y : z : w$. 
The \textit{infinitly distant plane} is the plane defined by the equation $w = 0$. 
We consider only algebraic surfaces in $\mathbb{R}^3$ or $\mathbb{CP}^3$, and for this reason we usually omit word ``algebraic''. 
From now on in this section, by $k^3$ we denote $\mathbb{R}^3$ or $\mathbb{CP}^3$. 

A surface $Z \in k^3$ is called \textit{ruled} if every point $p \in Z$ is incident to a line $\ell \in Z$. 
A surface is \textit{doubly ruled} if every its point of it is incident to two distinct lines contained in~\( Z \). It is well-known that if an algebraic surface in \( k^3 \) is doubly ruled, then it is a plane or a regulus. 
Ruled surfaces distinct from planes or reguli are called \textit{singly ruled}.

Our main tool to prove that some surface is ruled is the following theorem.

\begin{theorem}[Cayley--Salmon \cite{salmon}, Monge \cite{monge}] \label{Cayley-Salmon}
If $S \subset \mathbb{CP}^3$ is a surface of degree $d$, which does not contain irreducible ruled components, then there is a surface $T$ of degree at most $11d - 24$ such that $S$ and $T$ do not have common irreducible components and each line of $S$ is contained in $T$.
\end{theorem}

Under the conditions of Theorem~\ref{Cayley-Salmon}, the degree of the intersection curve of surfaces $S$ and $T$ does not exceed $d(11d - 24)$. Therefore, we obtain the following corollary. 

\begin{corollary} \label{bound on number of lines}
Let $S \subset \mathbb{CP}^3$ be a surface of degree $d$, which does not contain irreducible ruled components. Then $S$ contains at most $d(11d - 24)$ lines.
\end{corollary}

The \textit{absolute conic} in $\mathbb{CP}^3$ is given by the equations $x^2 + y^2 + z^2 = 0, w = 0$.

A \textit{complex circle} is an irreducible conic in $\mathbb{CP}^3$ having two distinct common points with the absolute conic. 
Clearly, a circle in $\mathbb{R}^3$ is a subset of a complex circle.




A line in $\mathbb{CP}^3$ can be naturally defined in the Pl\"ucker coordinates $\mathbb{CP}^5$: The line passing through points $x_1 : y_1 : z_1 : w_1$ and $x_2 : y_2 : z_2 : w_2$ is identified with the point
\[
x_1y_2 - x_2y_1 : x_1z_2 - x_2z_1 : x_1w_2 - x_2w_1 : y_1z_2 - y_2z_1 : y_1w_2 - y_2w_1 : z_1w_2 - z_2w_1\in \mathbb{CP}^5.
\]

Let $Z \subset k^3$ be an irreducible singly ruled surface.
A line $\ell \subset Z$ is called an \textit{exceptional} line of~$Z$  if every point of $\ell$ is incident to another line contained in~$Z$.
A point $p \in Z$ is called an \textit{exceptional} point of $Z$ if it is incident to infinitely many lines contained in $Z$.
An upper bound on the number of exceptional lines and points on a irreducible ruled surface in $\mathbb{CP}^3$ is proved in the following proposition (see \cite[Corollary 3.6]{guth-katz}).

\begin{proposition} \label{exceptional lines}
If $Z \subset \mathbb{CP}^3$ is an irreducible ruled surface 
different from a plane and a regulus, then $Z$ contains at most two exceptional lines and at most one exceptional point.
\end{proposition}

Lines that are not exceptional we call \textit{generators} or \textit{generating lines}.
To describe all generating lines contained in ruled surface, we use the following theorem (see \cite[Section II]{edge}). 

\begin{theorem} \label{parameterisation}
If $Z \subset \mathbb{CP}^3$ be an irreducible ruled surface 
different from a plane and a regulus, then all generating lines of $Z$ can be parameterized by an irreducible algebraic curve in the Pl\" ucker space. 
\end{theorem}

We also use the following property of generating lines
(see \cite[Lemma 1.3]{polo2011ruled}).

\begin{lemma} \label{intersection with everything}
If $Z \subset \mathbb{CP}^3$ be an irreducible ruled surface, then each algebraic curve on $Z$, which is not a generating line, intersects all generating lines of $Z$.
\end{lemma}


Our main tool to show Theorem \ref{lines-circles} is the following proposition, which we prove in Section \ref{proof of propositions}.


\begin{proposition} \label{main proposition}
	Let $\mathcal{L}$ be a collection of $n$ lines and $\mathcal{C}$ be a collection of $m$ circles in $\mathbb{R}^3$. 
	Suppose that for some $A \ge 10^5 \min(n, m)^{1/2}$, each curve of $\mathcal{L} \cup \mathcal{C}$ contains at least $A$ points of $\PP(\mathcal{L}, \mathcal{C})$.
	Then there are at most $500 \min(n, m) / A$ planes and hyperboloids with one sheet containing all curves of $\mathcal{L} \cup \mathcal{C}$.

\end{proposition}

\section{Bounding $\PP(\mathcal{L}, \mathcal{C})$} \label{proof of theorem}

\begin{proof}[Proof of Theorem \ref{lines-circles}]
Suppose that there are no $A$ curves of $\mathcal{L} \cup \mathcal{C}$ lying on a plane or on a hyperboloid with one sheet. 
Let us show that $|\PP(\mathcal{L}, \mathcal{C})| \le 1000 (n + m)A$.

If a curve $\omega \in \mathcal{L} \cup \mathcal{C}$ contains at most $A$ points of $\PP(\mathcal{L}, \mathcal{C})$, then $|\PP(\mathcal{L}\setminus \{\omega\}, \mathcal{C}\setminus\{\omega\})| \ge |\PP(\mathcal{L}, \mathcal{C})|- A$.
Removing one by one such curves \( \omega_1, \ldots, \omega_k \), one achieves the following scenario: Both collections $\mathcal L' := \mathcal{L} \setminus \{ \omega_1, \ldots, \omega_k \} $ and $\mathcal C' := \mathcal{C} \setminus \{ \omega_1, \ldots, \omega_k \}$ contain only curves with at least $A$ points of $\PP(\mathcal{L}', \mathcal{C}')$.
It is enough to show that the theorem holds for these collections. Indeed, 
\[
\PP(\mathcal{L}, \mathcal{C}) \le \PP(\mathcal{L}', \mathcal{C}') + kA \le 1000A(|\mathcal{L}'| + |\mathcal{C}'|) + A (|\mathcal{L}| + |\mathcal{C}| - |\mathcal{L}'| - |\mathcal{C}'|) \le 1000A(|\mathcal{L}| + |\mathcal{C}|).
\]
Thus without loss of generality, we can assume that each curve of $\mathcal L\cup\mathcal C$ contains at least $A$ points of $\PP(\mathcal L, \mathcal C)$.
By Proposition \ref{main proposition}, there exist $k \leq 500 \min(n, m) / A \le A / 10$ planes and hyperboloids with one sheet such that their union contains all curves of $\mathcal{L} \cup \mathcal{C}$. Denote these surfaces by $Z_1, \ldots, Z_k$.

Let $\mathcal{L}_i$ and $\mathcal{C}_i$ be the subcollections of $\mathcal{L}$ and $\mathcal{C}$ lying on $Z_i$ but not lying on $Z_j$ for all $j > i$. Clearly $\bigcup \mathcal{L}_i = \mathcal{L}$ and $\bigcup \mathcal{C}_i = \mathcal{C}$.
Let $\omega$ be any curve of $\mathcal{L}_i\cup\mathcal{C}_i$.
By B\' ezout's theorem, $\omega$ contains at most $4$ points of intersection with $Z_j$ for $j > i$.
Thus for all $i < j$ we have $|\PP(\mathcal{L}_i, \mathcal{C}_j)| + |\PP(\mathcal{L}_j, \mathcal{C}_i)| \leq 4(|\mathcal{L}_i| + |\mathcal{C}_i|) \le 4A$.
Since \(|\mathcal{L}_i| + |\mathcal{C}_i| \le A \), we have  \(|\PP(\mathcal{L}_i, \mathcal{C}_i)| \le 2 |\mathcal{L}_i| \cdot |\mathcal{C}_i| \le A^2\). 
Hence
\[
|\PP(\mathcal{L}_i, \mathcal{C}_i)| + \sum_{j > i} (|\PP(\mathcal{L}_i, \mathcal{C}_j)| + |\PP(\mathcal{L}_j, \mathcal{C}_i)|) < A^2 + 4k(|\mathcal{L}_i| + |\mathcal{C}_i|) \le A^2 + 4kA \le 2A^2
\]
for all $i$. Finally, we obtain
\[
|\PP(\mathcal{L}, \mathcal{C})| \le \sum_i \left(|\PP(\mathcal{L}_i, \mathcal{C}_i)| + \sum_{j > i} \left( |\PP(\mathcal{L}_i, \mathcal{C}_j)| + |\PP(\mathcal{L}_j, \mathcal{C}_i)| \right) \right) < 2kA^2 \le 1000 \min(n, m) A,
\]
which finishes the proof.
\end{proof}

\section{Covering $\mathcal{L}$ and $\mathcal{C}$ by a small number of quadrics} \label{proof of propositions}

First, we prove a weak version of Proposition \ref{main proposition}. 

\begin{proposition}  \label{weak proposition}
	Let $\mathcal{L}$ be a collection of $n$ lines and $\mathcal{C}$ be a collection of $m$ circles in $\mathbb{R}^3$. 
	Suppose that for some $A \ge 100 \min(n, m)^{1/2}$, each curve of $\mathcal{L}$ and $\mathcal{C}$ contains at least $A$ points of $\PP(\mathcal{L}, \mathcal{C})$.
	Then there is an algebraic surface of degree at most $500 \min(n, m) / A$, which contains all curves of $\mathcal{L}$ and $\mathcal{C}$.
\end{proposition}


\subsection{Constructing a surface} \label{proof of weak proposition}

First, we state and prove several simple observations that we need to show Proposition \ref{weak proposition}.

\begin{lemma} [Multiplicative Chernoff bound] 

Let $X_1, \ldots, X_n$ be independent Bernoulli random variables with parameter $p$. Then the following inequalities hold
\begin{enumerate}
	\item $\mathbb{P}\left(\frac{1}{n}\sum_{i=1}^n X_i \ge (1 + \delta)p\right) \le \left( \frac{e^\delta}{(1 + \delta)^{1 + \delta}} \right)^{np};$
	\item $\mathbb{P}\left(\frac{1}{n}\sum_{i=1}^n X_i \le (1 - \delta)p\right) \le \left( \frac{e^{-\delta}}{(1 - \delta)^{1 - \delta}} \right)^{np}.$	
\end{enumerate}

\end{lemma}

Substituting $\delta = 1$ in the first inequality and $\delta = \frac{1}{2}$ in the second one, we obtain the following corollary.

\begin{corollary} \label{Chernoff}
Let $X_1, \ldots X_n$ be independent Bernoulli random variables with parameter $p$. Then the following inequalities hold
\begin{enumerate}
	\item $\mathbb{P}\left(\frac{1}{n}\sum_{i=1}^n X_i \ge 2p\right) \le \exp \left( - \frac{1}{4} np \right);$
	\item $\mathbb{P}\left(\frac{1}{n}\sum_{i=1}^n X_i \le \frac{1}{2} p\right) \le \exp \left( -\frac{1}{8} np \right).$

\end{enumerate}

\end{corollary}

A trivial upper bound on the minimal degree of the surface containing all curves of \( \mathcal{L} \cup \mathcal{C} \) is shown in the following lemma.

\begin{lemma}  \label{trivial degree bound}
If $\mathcal{L}$ is a collection of $n$ lines in \( \mathbb{R}^3 \) and $\mathcal{C}$ be a collection of $m$ circles in $\mathbb{R}^3$, then there is a surface of degree at most $(12(n + m))^{1/2}$ containing all curves of \( \mathcal{L} \cup \mathcal{C} \).
\end{lemma}

\begin{proof}
Set $d =  \left \lfloor (12(n + m))^{1/2} \right \rfloor$. 
Choose $2d + 1$ pairwise distinct points on each curve of $\mathcal{L} \cup \mathcal{C}$. 
In total there are $k = (n + m)(2d + 1)$ chosen points. 
Since
$$k < 2(n + m)(d + 1) =  \frac{1}{6} \cdot 12(n + m) (d + 1) \le \frac{(d + 1)^3}{6} < \binom{d + 3}{3},$$
there is a surface $Z$ of degree $d$ containing all chosen points.
By B\' ezout's theorem and the fact that each curve \( \omega \in \mathcal{L} \cup \mathcal{C} \) intersects $Z$ at $2d + 1$ points, we obtain that $\omega$ lies on $Z$.
\end{proof}

Finally, we need the following trivial combinatorial lemma. 

\begin{lemma} \label{vertex-edge matching}  
If $G$ is a graph without isolated vertices, then there is a matching $M$ between $V(G)$ and $E(G)$ of size at least $|V(G)| / 2$ such that for each pair $(v, e) \in M$ vertex $v$ is endpoint of edge~$e$.

\end{lemma}

\begin{proof}	

Consider any connected component of $G$ and choose arbitrary vertex $v$ in this component. 
Let $d(u)$ be the distance from $v$ to $u$ in $G$, that is, the minimal number of edges in the path between $u$ and $v$. 
Let us match vertex $u \neq v$ to an edge $uw \in E(G)$ with property $d(u) = d(w) + 1$.
Obviously, all matched edges are different.
In each component of size $n$, there are at least $n - 1 \ge n / 2$ matched vertices. 
\end{proof}

\begin{proof}[Proof of Proposition \ref{weak proposition}]

The proof of Proposition \ref{weak proposition} is by induction on $n + m$.  

The base of induction \( n + m = 0 \) is trivial. 

The step of the induction can be done using the following lemma.

\begin{lemma} \label{small degree surface}
Let the conditions of Proposition \ref{weak proposition} hold.
\begin{enumerate}[topsep=-3pt,itemsep=-1ex,partopsep=1ex,parsep=1ex]
\item If \(n \le m\), then there is a surface of degree at most \(  D =  \lceil 200n / A \rceil \) containing at least~\( \lceil \frac{3n}{4} \rceil \) lines;
\item If \(n \ge m\), then there is a surface of degree at most \(  D =  \lceil 200m / A \rceil \) containing at least~\( \lceil \frac{3m}{4} \rceil \) circles.
\end{enumerate}

\end{lemma}

Lemma  \ref{small degree surface} is proved at the end of Subsection \ref{proof of weak proposition}.

By Lemma \ref{small degree surface}, there is a surface \( Z_0 \) of degree at most \(  D =  \lceil 200\min(n, m) / A \rceil \) containing at least three quarters of curves of a smaller collection among $\mathcal{L}$ and $\mathcal{C}$. 
Let $n_1 = |\mathcal{L}_1|$ and $m_1 = |\mathcal{C}_1|$, where $\mathcal{L}_1 \subset \mathcal{L}$ and $\mathcal{C}_1  \subset \mathcal{C}$ are the collections of lines and circles respectively that are not contained in $Z_0$.
If $n < m$, then $n_1 \le n/4$ and $4\min(n_1, m_1) \le 4n_1 \le n = \min(n, m)$. Similarly, the inequality $4\min(n_1, m_1) \le \min(n, m)$ holds in the case $m \ge n$.


Every line of \( \mathcal{L}_1 \) (every circle of \( \mathcal{C}_1 \)) intersects at least $A - 2\deg Z_0 \ge A - 2D > 0$ circles  of \( \mathcal{C}_1 \) (lines of \( \mathcal{L}_1 \)).
Thus there are two possible cases: $n_1 = m_1 = 0$ or $n_1 > 0$ and $m_1 > 0$. 
In the first case, we find the surface of degree at most $\lceil 200 \min(n, m) /A \rceil$ containing all curves of $\mathcal{L} \cup \mathcal{C}$.

Suppose $n_1 > 0$ and $m_1 > 0$. Each curve of $\mathcal{L}_1 \cup \mathcal{C}_1$ has at least $A_1 = A - 2D$ points of $\PP(\mathcal{L}_1, \mathcal{C}_1)$. Hence we get
$$|A_1| \ge \frac{1}{2}|A| \ge 50 \min(n, m)^{1/2} \ge 50 (4 \min(n_1, m_1))^{1/2} = 100 \min(n_1, m_1)^{1/2}$$

By the induction hypothesis, there is a surface of degree $500 \min(n_1, m_1) /A_1$ containing all curves of $\mathcal{L}_1$ and $\mathcal{C}_1$. 
The union of this surface and $Z_0$ is a surface of degree at most $\left \lceil 200 \min(n, m) / A \right \rceil + 500 \min(n_1, m_1) / A_1 \le 500n/A$ containing all curves of $\mathcal{L}$ and $\mathcal{C}$.
\end{proof}

\begin{proof}[Proof of Lemma \ref{small degree surface}]
Suppose $n \le m$. The opposite case $n \ge m$ is obtained from proof below by changing lines to circles and vice versa.

Let $p = \frac{D^2}{25n}$. 
Clearly, $D \ge 10$ and $p = \frac{D^2}{25n} \le \frac{(200n + A)^2}{25n \cdot A^2} \le \frac{(200n + 2n)^2}{25n \cdot A^2} < 1$. 

Let $\mathcal{L}_0 \subset \mathcal{L}$ be a random subset chosen by picking every line independently with probability~$p$.
Note that $\mathbb{E}|\mathcal{L}| = pn = \frac{D^2}{25}$. 
By Corollary \ref{Chernoff}, we obtain
$$\mathbb{P}\left(|\mathcal{L}_0| >  \frac{2D^2}{25}\right) = \mathbb{P}\left(|\mathcal{L}_0| > 2pn\right) \le \exp\left( -\frac{1}{4}pn  \right) = \exp\left( -\frac{D^2}{100}  \right) < \frac{1}{2}\text{.}$$

For each $q \in \PP(\mathcal{L}, \mathcal{C})$, let $X_q$ be the event $\{\exists \ell \in \mathcal{L}_0 :  q \in \ell \}$. 
Consider an arbitrary circle $\gamma$ in $\mathcal{C}$. 
Set $Q_\gamma = \PP(\mathcal{L}, \{\gamma\})$. 
By the premises of the proposition, $|Q_\gamma| \ge A$.
Let us match some points of $Q_\gamma$ to distinct lines passing through these points in the following way.
If a point $q \in Q_\gamma$ belongs to a line $\ell \in \mathcal{L}$ with $|\ell \cap \gamma| = 1$, then $q$ is matched to $\ell$.
Let $Q_\gamma^1$ be the set of points that are matched after previous operation and $Q_\gamma^2 = Q_\gamma \setminus  Q_\gamma^1$.
Let $G$ be the graph such that $V(G) = Q_\gamma^2$ and $E(G)$ is a set of pairs of points of $Q_\gamma^2$ belonging to one line of $\mathcal{L}$.
Since there are no isolated vertices in $G$, by Lemma \ref{vertex-edge matching}, distinct lines can be matched to at least half of the points of $Q_\gamma^2$.
Let $Q_\gamma'$ be the set of matched points of~$Q_\gamma$.
Clearly, $|Q_\gamma'| \ge |Q_\gamma^1| + |Q_\gamma^2| / 2 \ge   A / 2$.
For each $q \in Q_\gamma'$, let $X^\gamma_q$ be the event that the line corresponding to $q$ is contained in $\mathcal{L}_0$.
Clearly, for each $q \in Q_\gamma'$ we have $X^\gamma_q \subset X_q$, and thus $\mathbb{P}(X_q) \ge \mathbb{P}(X^\gamma_q) = p$. 
If $S_\gamma = |\PP(\mathcal{L}_0, \{\gamma\})|$ and $S_\gamma' = \sum_{q \in Q_\gamma'} \mathbf{I}_{X^\gamma_q}$, then 
\[
S_\gamma = \sum_{q \in Q_\gamma} \mathbf{I}_{X_q} \ge \sum_{q \in Q_\gamma'} \mathbf{I}_{X^\gamma_q} = S_\gamma'
\]
and
\[
\mathbb{E}S_\gamma' =
\sum_{q \in Q_\gamma'} \mathbb{P}(X_q) \ge 
\sum_{q \in Q_\gamma'} \mathbb{P}(X^\gamma_q) \ge 
|Q_\gamma'| \cdot p \ge 
Ap / 2 \text{.}
\]

Since events $X^\gamma_q$ are independent, by Corollary \ref{Chernoff}, we obtain
\[
\mathbb{P}\left(S_\gamma \le 2D\right) \le 
\mathbb{P}\left(S_\gamma' \le 2D\right) \le 
\mathbb{P}\left(S_\gamma' \le \frac{AD}{100n} D \right) = 
\mathbb{P}\left(S_\gamma' \le \frac{1}{4}Ap\right) \le 
\exp\left(-\frac{1}{8} Ap \right) =
\]
\[
= \exp\left(-\frac{AD^2}{200n} \right) \le 
\exp \left( -\frac{200n}{A} \right) \le e^{-100} \le 
\frac{1}{16}\text{.}
\]

Consider any line $\ell \in \mathcal{L}$.
Let $\mathcal{C}_\ell$ be a maximum size family of circles of $\mathcal{C}$ such that it is possible to match all circles of $\mathcal{C}_{\ell}$ to different points of $\PP(\{\ell\}, \mathcal{C})$. 
Similarly, using Lemma \ref{vertex-edge matching}, we obtain $|\mathcal{C}_l| \ge A/2$.
Let $T_{\ell}$ be the number of circles of $\mathcal{C}_{\ell}$ such that each of them contains at most $2D$ intersection points with lines of $\mathcal{L}_0$. We have

$$\mathbb{E} T_{\ell} = \sum_{\gamma \in \mathcal{C}_{\ell}} \mathbb{P}\left(S_\gamma \le 2D\right) \le \frac{|\mathcal{C}_{\ell}|}{16}\text{.}$$

By Markov's inequality and the inequality $|\mathcal{C}_\ell| \ge A / 2 \ge 1000n / A \ge 4D$, we get 
$$\mathbb{P}(T_{\ell} \ge |\mathcal{C}_{\ell}| - 2D) \le \frac {\mathbb{E}T_{\ell}}{|\mathcal{C}_{\ell}| - 2D} \le \frac {|\mathcal{C}_{\ell}|}{16(|\mathcal{C}_{\ell}| - 2D)} \le \frac{1}{8}\text{.}$$ 

Let $\mathcal{L}'$ be the collection of lines $\ell \in \mathcal{L}$ such that $T_{\ell} \ge |\mathcal{C}_{\ell}| - 2D$. By Markov's inequality, we have
$$\mathbb{P} \left( |\mathcal{L}'| > \frac{n}{4} \right) \le \frac{4\mathbb{E}|\mathcal{L}'|}{n} = \frac{4}{n} \sum \limits_{\ell \in \mathcal{L}} \mathbb{P}(T_{\ell} \ge |\mathcal{C}_{\ell}| - 2D) \le \frac{1}{2}\text{.}$$

Therefore, there is $\mathcal{L}_0 \subset \mathcal{L}$ such that $|\mathcal{L}_0| \le \frac{2D^2}{25}$ and $|\mathcal{L}'| \le n / 4$. 
Thus, by Lemma \ref{trivial degree bound}, there is a surface of degree at most $D$ containing all lines of $\mathcal{L}_0$. 
Let $Z_0$ be a surface of minimal degree containing every line of $\mathcal{L}_0$.
By B\' ezout's theorem, every line or circle intersecting $Z_0$ in at least $2D + 1$ points lies on $Z_0$.
Since each line $\ell$ of $\mathcal{L} \setminus \mathcal{L}'$ has at least $2D + 1$ intersection points with circles that has at least $2D + 1$ intersections with $Z_0$, the line $\ell$ lies on $Z_0$.
\end{proof}

\subsection{Decomposition into quadrics}

In this section we show Proposition \ref{main proposition}. From now on we assume that all polynomials have coefficients in \(\mathbb{R}\) and all surfaces are considered in \(\mathbb{CP}^3\) except otherwise is explicitly indicated.
For a polynomial $Q \in \mathbb{R}[x, y, z]$, let $Z(Q)$ be the surface in $\mathbb{CP}^3$ defined by the equation $Q = 0$. 

\begin{proof}[Proof of Proposition \ref{main proposition}] 
Let $Q$ be the polynomial of minimum degree such that $Z(Q)$ contains all curves of $\mathcal{L} \cup \mathcal{C}$. 
Among all polynomials of minimum degree, we choose one with maximum number of irreducible over \(\mathbb{R}\) components.
By Proposition \ref{weak proposition}, we have $d \le 500 \min(n, m) / A \le A / 100$. 

Let $Q_1, \ldots, Q_k$ be the irreducible over \(\mathbb{R}\) components of $Q$. 
Suppose that \(Q_i\) is reducible over~\(\mathbb{C}\), that is, \(Q_i = Q_{i, 1} \cdot Q_{i, 2}\), where \(Q_{i, j} \in \mathbb{C}[x, y, z] \). 
If we replace all coefficients in \(Q_{i, j}\) be their real part, then we obtain polynomial \(Q_{i, j}' \in \mathbb{R}[x, y, z]\) such that \(\deg Q_{i, j}' \leqslant \deg Q_{i, j}\) and \(Q_{i, j}'\) contains all curves of \(\mathcal{L} \cup \mathcal{C}\) lying in \(Q_{i, j}\).
Thus replacement \(Q_i\) with \(Q_{i, 1}' \cdot Q_{i, 2}'\) decreases degree of the polynomial or increases the number of irreducible over \(\mathbb{R}\) components, a contradiction.

Let $\mathcal{L}_i \subseteq \mathcal{L}$ and $\mathcal{C}_i \subseteq \mathcal{C}$ be subcollections of curves contained in $Z(Q_i)$, but not contained in $Z(Q_j)$ for all $j \neq i$.
Since $Q$ is the polynomial of minimum degree, $\mathcal{L}_i \cup \mathcal{C}_i \neq \varnothing$.
Let $A_i$ be the maximum number~$k$ such that every curve of $\mathcal{L}_i \cup \mathcal{C}_i$ contains at least $k$ points of $\PP(\mathcal{L}_i, \mathcal{C}_i)$.
By B\' ezout's theorem, each curve of $\mathcal{L}_i \cup \mathcal{C}_i$ has at most $2d$ intersection points with $\bigcup_{j \neq i} Z(Q_j)$.
Hence each curve of $\mathcal{L}_i \cup \mathcal{C}_i$ contains at most $2d$ points of $\PP(\mathcal{L}, \mathcal{C}) \setminus \PP(\mathcal{L}_i, \mathcal{C}_i)$ and $A_i \ge A - 2d \ge \frac{9}{10}A$.

Denoting $n_i = |\mathcal{L}_i|$, $m_i = |\mathcal{C}_i|$ and $d_i = \deg Q_i$, by Propositions \ref{weak proposition}, we have
$$d_i \le \frac{500 \min(n_i, m_i)}{A_i} \le \frac{1000 \min(n_i, m_i)}{A} \le \frac{1000 \min(n_i, m_i)}{5000 (\min(n, m))^{1/2}} \le \frac{(\min(n_i, m_i))^{1/2}}{5}.$$

Using inequality $11d_i^2 \le \frac{11}{25} \min(n_i, m_i) < n_i$ and Corollary \ref{bound on number of lines}, we get that $Z(Q_i)$ is ruled in $\mathbb{CP}^3$.

Suppose that $Z(Q_i)$ contains the absolute conic. 
By Lemma \ref{intersection with everything}, each generating line passes through the absolute conic. 
Since the absolute conic does not contain real points and every line of $\mathcal{L}_i$ is real, these lines do not intersect the absolute conic and thus are exceptional. 
By Proposition~\ref{exceptional lines}, the surface $Z(Q_i)$ contains at most two exceptional lines, but $n_i \ge (5d_i)^{1/2} > 2$. 
Hence $Z(Q_i)$ does not contain the absolute conic.

Let $S_i$ be the set of intersection points of $Z(Q_i)$ and the absolute conic. 
By B\' ezout's theorem, the inequality $|S_i| \le 2d_i$ holds. 

Each circle of $\mathcal{C}_i$ is a subset of some complex circle in \(\mathbb{CP}^3\). 
Each of these complex circle intersects the absolute conic in exactly two points, that is, contains two points of $S_i$. 
Thus there are at least
$$\frac{|\mathcal{C}_i|}{\binom{|S_i|}{2}} \ge \frac{2|\mathcal{C}_i|}{|S_i|^2} \ge \frac{m_i}{2d_i^2} \ge 5$$
complex circles passing through the same pair of points of $S_j$. 
Let $\gamma_1, \ldots, \gamma_5 \in \mathcal{C}_i$ be the circles in $\mathbb{R}^3$ corresponding to these complex circles.
Since for any $i, j \in \{1, \ldots, 5\}$ complex circles corresponding to $\gamma_i$ and $\gamma_j$ have the same intersection points with the infinite plane, $\gamma_i$ and $\gamma_j$ lie in the same plane or in parallel planes. 

Lemma~\ref{intersection with everything} applied to the irreducible ruled surface $Z(Q_i)$ implies that each generating line passes through each \(\gamma_i\).
Suppose that there are two circles lying in one plane \( P \). 
For each point \( p \) of these circles, but, probably, the points of intersection, there is a line passing \( p \) and contained in \( P \) and in \( Z(Q_i) \). 
Since infinitely many lines lies in the intersection of \( P \) and \( Z(Q_i) \), plane \( P \) is contained in \( Z(Q_i) \) and hence \( Q_i \) is linear.

From now on, we assume that \(\gamma_i\) lies in distinct real planes. Let $\mathcal{X}$ be the family of polynomials of the form $h(x, y, z) = (x^2 + y^2) \cdot p_1(z) + x \cdot p_2(z) + y \cdot p_3(z) + p_3(z) \in \mathbb{R}[x, y, z]$, where $\deg h \le 4$.
This family is a vector space of dimension 16 and thus there is a polynomial of this family passing through any 15 points in the space. 


Assume that the planes containing circles $\gamma_i$ are parallel to $Oxy$.
Choose arbitrary three points on each circle $\gamma_i$. 
There is a polynomial $Q_i' \in \mathcal{X}$ such that $Z(Q_i')$ passes through chosen points. 
Since the intersection of $Z(Q_i')$ with a plane parallel to $Oxy$ is a circle, a line or the empty set, this surface contains each of $\gamma_i$. 
Moreover, each generating line of $Z(Q_i)$ intersects $Z(Q_i')$ in at least five points and thus is contained in it. 
Therefore, $Q_i$ divides $Q_i'$. 

If $Q_i'$ is irreducible over \(\mathbb{R}\), then $Q_i \in \mathcal{X}$. 
Suppose to the contrary that $Q_i' = f \cdot g$, where \(f, g \in \mathbb{R}[x, y, z]\) and $\deg f > 0, \deg g > 0$. 
Suppose that both $f$ and $g$ depends on variables $x$ or~$y$.
In this case, these polynomials can be expressed in form $f(x, y, z) = x \cdot u_1(z) + y \cdot u_2(z) + u_3(z)$ and $g(x, y, z) = x \cdot v_1(z) + y \cdot v_2(z) + v_3(z)$.
From these representations follows that $u_1 v_1 = u_2 v_2$ and $u_1 v_2 + u_2 v_1 = 0$. 
It means that $u_1 = u_2 = 0$ or $v_1 = v_2 = 0$, a contradiction.
If $Q_i'$ cannot be expressed as a product of polynomials, which depends on $x$ or $y$, then each factor of $Q_i'$ lies in $\mathcal{X}$.
Hence $Q_i \in \mathcal{X}$. 

If $Q_i$ has degree at most two, then it is a plane or a hyperboloid with one sheet. 
Otherwise, the line $\ell \in \mathbb{CP}^3$ defined by $z = w = 0$ lies on $Z(Q_i)$, because $\deg Q_i \geq 3$ and degree on variables \( x \) and \( y \) in \( Q_i \) does not exceed two. 
Consider the following two cases.

\textit{Case 1:} $\ell$ is exceptional.
Since planes containing \(\ell\) are parallel to \(Oxy\), every line intersecting \(\ell\) is parallel to \(Oxy\).
A line that is parallel to \(Oxy\) cannot intersect all circles $\gamma_i$ and thus it cannot be a generating line.
But by Lemma \ref{intersection with everything}, for each point of $\ell$ there is a generating line passing through it,  a contradiction.

\textit{Case 2:} $\ell$ is generating. 
Let $I$ be the intersection of $\ell$ and the absolute conic.
By Theorem \ref{parameterisation}, the generating lines form an algebraic curve in Pl\"ucker space, and, therefore, there is a sequence of generating lines $\ell_i$ converging to $\ell$. 
As in the Case 1, there are no generating lines passing through \(\ell\), thus $\ell_i \cap I = \varnothing$.
For each $j \in \{1, \ldots, 5\}$, the line $\ell_i$ intersects $\gamma_j$ at some point $P_i^j$.
Each of the five points $P_i^j$ converges to one of the two points of the set $I$. 
We may assume that $P_i^1$ and $P_i^2$ converge to the same point $P \in I$. 
For \( j \in \{1, 2\}\), the sequence of lines \(PP^j_i\) converges to tangent line to \(\gamma_i\) at \(P\).
Since $\gamma_1$ and $\gamma_2$ are not coplanar, these tangent lines are distinct and the plane $PP_i^1P_i^2$  converges to the projective plane $\Omega$ containing them. 
The projective plane $\Omega$ has a unique common point with $\gamma_1$, while $\ell \subset \Omega$ intersects $\gamma_1$ by the two-point set $I$, a contradiction.

Therefore, for each $i$, the degree of $Q_i$ does not exceed two. 
\end{proof}

\section*{\large Acknowledgments}
The author acknowledge the financial support from the Ministry of Education and Science of the Russian Federation in the framework of MegaGrant no 075-15-2019-1926.
The author is grateful to Alexandr Polyanskii for fruitful and inspirational discussions, to Alexey~Balitskiy for numerous valuable comments that helped to significantly improve the presentation of the paper and to Mikhail~Skopenkov for advices on references.

\bibliographystyle{siam}
\bibliography{mybib}{}
\end{document}